\documentclass[12pt, a4paper]{amsart}
\usepackage[utf8]{inputenc}
\usepackage{amsthm}
\usepackage{mathrsfs}
\usepackage[none]{hyphenat}
\usepackage[english]{babel}
\usepackage{setspace}
\usepackage{pinlabel}    
\usepackage{mathabx}

\usepackage[a4paper, top=3cm, bottom=4cm,left=3cm,right=3cm]{geometry}

\usepackage{comment}

\usepackage{pgf,tikz}
\usepackage{microtype}

\usepackage{tikz-cd}
\usepackage{tikz,pgfplots}
\usetikzlibrary{positioning, angles}
\usetikzlibrary{arrows}
\usetikzlibrary{decorations.pathmorphing, decorations.pathreplacing}
\usetikzlibrary{decorations.markings}
\usetikzlibrary{calc,shapes}
\usepackage{mathtools}
\usepackage{lipsum}
\usepackage{comment}
\usepackage{float}    
\usetikzlibrary{shapes.geometric}
\usepackage{tikz}
\usepackage{graphicx}
\usepackage{amsmath,amssymb}
\let\phi\varphi

\usepackage{enumerate}
\usepackage{caption}
\usepackage{subcaption}

\usepackage{amsmath, amsthm, amssymb, amsfonts, xfrac}
\newtheorem{prop}{\textsc{Proposition}}[section]
\newtheorem{cor}[prop]{\textsc{Corollary}}
\newtheorem{thm}[prop]{\textsc{Theorem}}

\newtheorem{lemma}[prop]{\textsc{Lemma}}
\newtheorem*{thm*}{\textsc{Theorem}}

\theoremstyle{definition}
\newtheorem{remark}[prop]{\textsc{Remark}}
\newtheorem{defn}[prop]{\textsc{Definition}}

\newcommand{\Z}{\mathbb{Z}}
\newcommand{\Q}{\mathbb{Q}}

\newcommand{\g}{\mathfrak{g}}
\newcommand{\C}{\mathcal{C}}
\newcommand{\lie}{\mathfrak{l}}

\DeclareMathOperator{\Der}{Der}
\DeclareMathOperator{\sDer}{sDer}
\DeclareMathOperator{\tDer}{tDer}

\DeclareMathOperator{\Aut}{Aut}

\DeclareMathOperator{\lk}{lk}

\DeclareMathOperator{\Diffeo}{Diffeo}

\DeclareMathOperator{\Fix}{Fix}

\let\epsilon\varepsilon

\title[Solvability of concordance groups and Milnor invariants]{Solvability of concordance groups and Milnor invariants}
\author[]{Alessio Di Prisa, Giovanni Framba}
\begin{document}
\maketitle
\begin{abstract}
Using Milnor invariants, we prove that the concordance group $\C(2)$ of $2$-string links is not solvable. As a consequence we prove that the equivariant concordance group of strongly invertible knots is also not solvable, and we answer \cite[Conjecture 1.3]{kuzbary2023note}.
\end{abstract}

\section{Introduction}

Recall that a strongly invertible knot is given by a pair $(K,\rho)$, where $K \hookrightarrow S^3$ is a knot and $\rho:S^3 \rightarrow S^3$ is a smooth orientation preserving involution such that $\rho(K) = K$ and $\rho$ reverses the orientation on $K$.
In \cite{sakuma}, Sakuma introduced the definitions of equivariant connected sum and equivariant concordance for (directed) strongly invertible knots. Using this notions, he defines the group $\widetilde{\C}$ of equivariant concordance classes of strongly invertible knots.
Following Sakuma, strongly invertible knots have been a classic object of study. Recently, there has been a strong new interest in this topic \cite{collari2023strongly,hirasawa2022invariant,boyle2023classification,lobb2021refinement,lipshitz2022khovanov}, and in particular towards the side of equivariant concordance \cite{alfieri2021strongly,boyle2021equivariant,dai_mallick_stoffregen,di2023new, di2023equivariant,miller_powell}.
However, very little is known concerning the group structure of $\widetilde{\C}$, and even if the first author proved in \cite{diprisa} that $\widetilde{\C}$ is not abelian, a lot of questions remain open. For instance, Alfieri and Boyle \cite{alfieri2021strongly} conjecture that the equivariant concordance group contains a non-abelian free group.

In this paper we push the investigation about how far $\widetilde{\C}$ is from being abelian a step further. In fact, in Section~\ref{sec:solvab} we prove the following:

\vspace{3.5mm}
\noindent{\bf Theorem \ref{thm:main}} {\it The equivariant concordance group of strongly invertible knots $\widetilde{\C}$
     is not solvable.}
\vspace{3.5mm}

The proof of Theorem \ref{thm:main} relies on the relation between strongly invertible knots, \emph{theta curves} and \emph{string links}, which we outline in Section \ref{sec:string}.
Recall that a theta curve is the embedding in $S^3$ of a graph with two vertices and tree edges joining them. An $n$-string link is (the image of) a proper embedding $\sigma$ of $\coprod_{k=1}^n I_i$, the disjoint union of $n$ copies of the interval $I=[0,1]$, in $D^2\times I$ such that $\sigma|_{I_i}(j)\in D^2\times\{j\}$ for each $i$ and for $j=0,1$.
Analogously to the knot case, it is possible to give natural definitions of sum and cobordism/concordance for theta curves and string links. The set of cobordism classes of theta curves forms a group denoted by $\Theta$, while $n$-string links up to concordance form a group denoted by $\C(n)$.
One can easily check that $\C(1)$ coincides with the knot concordance group, and hence it is an abelian group.
As shown by Le Dimet \cite{le1988cobordisme} the pure braid group $\mathcal{P}(n)$ naturally injects in $\C(n)$. Since for $n\geq3$ $\mathcal{P}(n)$ contains a non-abelian free subgroup, we easily get that $\C(n)$ in not abelian and in particular not solvable for $n\geq 3$. The same argument does not work in the case $n=2$, since $\mathcal{P}(2)\cong\Z$ is a central subgroup of $\C(2)$.
De Campos \cite{de2005boundary} proved that $\Theta$ and $\C(2)$ are related by the following theorem.

\vspace{3.5mm}
\noindent{\bf Theorem \ref{thm:garoto}} \cite[Proposition 2]{de2005boundary} {\it There is a split extension of groups:
  \[ 1 \rightarrow \mathcal{P}(2) \rightarrow \C(2) \rightarrow \Theta \rightarrow 1. \]
  }

Miyazaki in \cite{miyazaki} provided a proof that the group $\Theta$ is not commutative (which would imply that $\C(2)$ is not abelian), appealing on a result of Gilmer \cite{gilmer}. However, Friedl \cite{friedl} found gaps in the proof of the result in \cite{gilmer}.
Another proof of the fact that $\C(2)$ is not abelian can be found in \cite[Theorem 1.8]{meilhan2014abelian}.
In \cite[Theorem 1.1]{kuzbary2023note} the author proves that for all $n\geq 2$ the group $\C(n) / \langle\langle \mathcal{P} \rangle \rangle$ is not abelian and conjectures that it is not solvable either \cite[Conjecture 1.3]{kuzbary2023note}.

Using Milnor invariants \cite{milnor1954isotopy,habegger1998link}, we prove in Theorem \ref{thm:Cu_no_sol} that a certain subgroup $\C_u(2)$ (see Section \ref{sec:solvab}) of $\C(2)$ is not solvable. From this result we deduce Theorem~\ref{thm:main}.
Moreover, the non-solvability of $\C_u(2)$ implies the following theorem, which answers the question posed by Kuzbary.

\vspace{3.5mm}
\noindent{\bf Theorem \ref{thm:main.alpha}} {\it The quotient group $\C(n) / \langle\langle \mathcal{P}(n) \rangle\rangle$, of the $n$-strands string links over the normal closure of the pure $n$-braids subgroup, is not solvable for any $n\geq 2$.
In particular, the cobordism group of theta curves $\Theta\cong\C(2) / \mathcal{P}(2)$ is not solvable.}
\vspace{3.5mm}

\subsection*{Content of the paper}
In Section~\ref{sec:GrpLie} we provide the basic notions and definitions regarding the algebraic tools we need in Section~\ref{sec:prova}.   
Section~\ref{sec:string} contains a recap on strongly invertible knots and string links. Moreover, we recall the definition of Milnor invariants.
In Section~\ref{sec:solvab} we show how Theorem~\ref{thm:main} and Theorem~\ref{thm:main.alpha} are implied by Theorem~\ref{thm:Cu_no_sol}. 
Lastly, Section~\ref{sec:prova} is devoted to the proof of Theorem~\ref{thm:Cu_no_sol}.

\section{Groups and Lie algebras}\label{sec:GrpLie}
In this section we introduce the notation and the preliminary results that we will need in Section \ref{sec:prova}. For details on filtrations see \cite[Section 1]{darne2020braids}.
\begin{defn}
Let $G$ be a group. A \emph{filtration} $G_*=\{G_k\}_{k\geq1}$ on $G$ is a sequence of subgroups of $G$ such that
\begin{itemize}
    \item $G_1=G$,
    \item $G_{k+1}\subset G_{k}$,
    \item $[G_i,G_j]\subset G_{i+j}$.
\end{itemize}
\end{defn}

\begin{defn}
The \emph{graded Lie ring} associated with a filtration $G_*$ on a group $G$ is defined as the graded group
$$
\mathcal{L}(G_*)=\bigoplus_{k\geq 1}G_k/G_{k+1}
$$
with the Lie bracket induced by the commutator operation on $G$.
\end{defn}

\begin{defn}
Let $G_*$ and $H_*$ be filtrations on $G$ and $H$ respectively. An \emph{action} of $G_*$ on $H_*$ is a homomorphism $\phi:G\longrightarrow\Aut(H)$ such that for all $i,j\geq1$, $[G_i,H_j]\subset H_{i+j}$, where $[g,h]:=\phi(g)(h)h^{-1}$, for $g\in G$, $h\in H$.
\end{defn}

Given an action $\phi:G\longrightarrow \Aut(H)$ of $G_*$ on $H_*$, we have an induced homomorphism of Lie rings
$$
J_\phi:\mathcal{L}(G_*)\longrightarrow\Der(\mathcal{L}(H_*))
$$
$$
\overline{g}\longmapsto [g,-]
$$
called the \emph{Johnson homomorphism}, where $\Der(\mathcal{L}(H_*))$ is the Lie ring of derivations on $\mathcal{L}(H_*)$.

\begin{lemma}\cite[Proposition-definition 1.8]{darne2020braids}\label{lemma:induced_filtration}
Let $\phi:G\longrightarrow\Aut(H)$ be a homomorphism and let $H_*$ be a filtration on $H$. Then, the filtration $G_*$ given by
$$
G_i=\{g\in G\;|\;\forall j\geq1,\; [g,H_j]\subset H_{i+j}\}
$$
is the unique maximal filtration on some subgroup $G_1\subset G$ such that $\phi_{|G_1}:G_1\longrightarrow\Aut(H)$ is an action of $G_*$ on $H_*$.
\end{lemma}

\begin{lemma}\label{lemma:group_solv}
Let $G_*$ be a filtration on a group $G$. If $G$ is a solvable group then $\mathcal{L}(G_*)$ is a solvable Lie ring.
\begin{proof}
We denote by $G^{(n)}$ the derived series of $G$, and we denote $G^{(n)} \cap G_*$ the filtration on $G^{(n)}$ induced by the one on $G$.
We begin by showing that $\mathcal{L}(G_*)^{(1)} = [\mathcal{L}(G_*),\mathcal{L}(G_*)]$ is contained in $\mathcal{L}(G^{(n)} \cap G_*)$.
Fix a degree $k\geq 1$ of the filtration, then by definition:
\begin{align*}
    [\mathcal{L}(G_*),\mathcal{L}(G_*)]_k &\subseteq \bigplus_{i+j=k} 
    \left[ G_i/G_{i+1}, G_j/G_{j+1} \right]\\
    &\subseteq \frac{\langle\, [G_i,G_j] \;|\; i+j=k \,\rangle}{G_{k+1} \cap G^{(1)}}  \\
    &\subseteq \frac{G_k \cap G^{(1)}}{G_{k+1} \cap G^{(1)}}.
\end{align*} 
By iteration we obtain:
\[ \mathcal{L}(G_*)^{(n)} \subseteq \mathcal{L}(G^{(n)} \cap G_*). \]
Since $G$ is solvable, for a large enough $n$ the right hand side will be zero, hence the thesis.
\end{proof}
\end{lemma}

\begin{remark}
Let $\mathfrak{g}=\bigoplus_{k}\mathfrak{g}_k$ be a graded Lie algebra. Then the degree completion of $\g$ is given by
$$
\overline{\g}=\prod_{k\geq 1}\g_k.
$$
\end{remark}

\begin{lemma}\label{lemma:solv_completion}
Let $\g$ be a graded Lie algebra. Then $\g$ is solvable if and only if its degree completion $\overline{\g}$ is solvable.
\begin{proof}
    We have that $\g \subseteq \overline{\g}$, so if $\g$ is not solvable, $\overline{\g}$ cannot be solvable.
    Viceversa, assume that $\overline{\g}$ is not solvable and consider, for all $n\in\mathbb{N}$, the ideals $I_n=\bigoplus_{k\geq n} \g_k$ and $\overline{I}_n = \prod_{k\geq n} \overline{\g}_k$.  

    Firstly, we want to prove that, for $n$ large enough, the quotient algebra $\overline{\g} / \overline{I}_n$ is not solvable.
    Let $\pi_n:\overline{\g} \rightarrow \overline{\g} / \overline{I}_n$ be the projection. Denote by $\mathfrak{h}^{(n)}$ the $n$-th term of the derived series of a Lie algebra $\mathfrak{h}$. We have that for all $k$ and $n$: 
    \[ \pi_n\left(\overline{\g}^{(k)}\right)=\left( \overline{\g} / \overline{I}_n \right)^{(k)}.  \]
    Since $\overline{\g}$ is not solvable, for all $k\geq 1$ there exist an $x_k\in\overline{\g}^{(k)}$, $x\neq 0$. Let $n_0$ be the degree of the first non-zero component of $x$.
    We have that $\pi_{n_0+1}(x)\neq 0$, it follows that $\overline{\g} / \overline{I}_{n_0+1}$ is not solvable.
    Observe that $\overline{\g} / \overline{I}_n = \g / I_n$, hence $\g$ is not solvable, since it has a non-solvable quotient.
\end{proof}
\end{lemma}

In the following we will denote by $\lie_2$ the free graded Lie ring on two generators $X$ and $Y$ in degree 1.\iffalse
The homogeneous components of $\lie_2$ are given by:
\begin{itemize}
    \item $(\lie_2)_1=\langle X, Y\rangle$,
    \item $(\lie_2)_{k+1}=[(\lie_2)_1,(\lie_2)_k]$.
\end{itemize}\fi

\begin{defn}\label{def:sder}
We say that a derivation $D\in\Der(\lie_2)$ is \emph{tangential} if there exist $U,V\in\lie_2$ such that
$$
D(X)=[U,X]\qquad D(Y)=[V,Y].
$$
A tangential derivation $D$ is called \emph{special} (or \emph{normalized}) if $D(X+Y)=0$.
We will denote by $\tDer(\lie_2)$ and $\sDer(\lie_2)$ the sets of tangential derivations and special derivation respectively. It is not difficult to check that $\tDer(\lie_2)$ and $\sDer(\lie_2)$ are Lie subrings of $\Der(\lie_2)$.

In a similar way one can define the subalgebra of tangential and special derivation of the Lie algebra $\lie_2^\Q$ and of its degree completion $\overline{\lie_2^\Q}$ (see also \cite[Section 3]{alekseev2012kashiwara}).
\end{defn}

Observe that since $\lie_2$ is graded, the Lie ring $\sDer(\lie_2)$ naturally inherits a graded structure.

\begin{lemma}\label{lemma:sder_tensor}
The following map induces an isomorphism of graded Lie rings
$$
\sDer(\lie_2)\otimes\Q\longrightarrow\sDer(\lie_2^\Q)
$$
$$
D\otimes p/q\longmapsto p/q\cdot D.
$$
\end{lemma}
\begin{proof}
It is not difficult to check that the inverse of the map above can be constructed as follows. Let $D\in\sDer(\lie_2^\Q)$ and let $U,V\in\lie_2^\Q$ such that $D(X)=[U,X]$ and $D(Y)=[V,Y]$.
Then there exists $n$ large enough so that $nU,nV\in\lie_2\subset\lie_2^\Q$. Then clearly $nD$ restricts to a special derivation of $\lie_2$, and we get the inverse map as
$$
\sDer(\lie_2^\Q)\longmapsto\sDer(\lie_2)\otimes\Q
$$
$$
D\longmapsto nD\otimes 1/n.
$$
\end{proof}

\begin{lemma}\label{lemma:sder_completion}
The degree completion $\overline{\sDer(\lie_2^\Q)}$ of the special derivations of $\lie_2^\Q$ is naturally isomorphic to $\sDer(\overline{\lie_2^\Q})$.
\end{lemma}
\begin{proof}
Take $D\in\overline{\sDer(\lie_2^\Q)}$. We can see $D$ as a formal sum $D=\sum_{i\geq 1}D_i$, where $D_i\in\sDer(\lie_2^\Q)_i$.
Then $D$ acts as a derivation on $\overline{\lie_2^\Q}$ as follows. Given $Z\in\overline{\lie_2^\Q}$, we can write it as $Z=\sum_{i\geq1}Z_j$, with $Z_j\in(\lie_2^\Q)_j$. Then we define $D(Z)$ as
$$
D(Z)=\sum_{i,j\geq 1}D_i(Z_j)
$$
which is a well defined element of $\overline{\lie_2^\Q}$, since for every $n$ the number of terms $D_i(Z_j)$ of degree $\leq n$ is finite.
For every $i\geq 1$ take $U_i, V_i\in(\lie_2^\Q)_i$ such that $D_i(X)=[U_i,X]$ and $D_i(Y)=[V_i,Y]$, and let $U=\sum_{i\geq1}U_i,\;V=\sum_{i\geq1}V_i\in\overline{\lie_2^\Q}$.
Then it is immediate to see that $D(X)=[U,X]$ and $D(Y)=[V,Y]$, hence $D$ is a tangential derivation of $\overline{\lie_2^\Q}$.
Observe that since $D_i(X+Y)=0$ for every $i$, we get that $D(X+Y)=0$.
Therefore we have a map $\overline{\sDer(\lie_2^\Q)}\longrightarrow\sDer(\overline{\lie_2^\Q})$. Its inverse is given as follows.
Take $D\in\sDer(\overline{\lie_2^\Q})$. Since it is a tangential derivation there exist $U=\sum_{i\geq1}U_i,\;V=\sum_{i\geq1}V_i\in\overline{\lie_2^\Q}$, $U_i,V_i\in(\lie_2^\Q)_i$ such that $D(X)=[U,X]$ and $D(Y)=[V,Y]$.
Since $D$ is special, one can easily see that $[U_i,X]+[V_i,Y]=0$ for every $i$, by inspecting $D(X+Y)$ in each degree.
Let $D_i$ be the derivation of $\lie_2^\Q$ defined by $D_i(X)=[U_i,X]$ and $D_i(Y)=[V_i,Y]$, and observe that $D_i\in\sDer(\lie_2^\Q)_i$.
Hence we get an element $D=\sum_{i\geq1}D_i\in\overline{\sDer(\lie_2^\Q)}$.
\end{proof}

\begin{prop}\label{prop:solv_sder}
The Lie ring $\sDer(\lie_2)$ is not solvable.
\end{prop}
\begin{proof}
First of all observe that $\sDer(\lie_2)$ is solvable if and only if $\sDer(\lie_2)\otimes\Q$ is solvable, where $\sDer(\lie_2)\otimes\Q\cong\sDer(\lie_2^\Q)$ by Lemma \ref{lemma:sder_tensor}.
Therefore, using Lemma \ref{lemma:sder_completion} and \ref{lemma:solv_completion} it is sufficient to show that $\sDer(\overline{\lie_2^\Q})$ is not solvable,
By \cite[Theorem 4.1]{alekseev2012kashiwara} we know that $\sDer(\overline{\lie_2^\Q})$ contains the Grothendieck-Teichmuller algebra $\mathfrak{grt}$ (see \cite[Section 4.2]{alekseev2012kashiwara}) as a Lie subalgebra. In \cite{brown2012mixed} Brown proved that $\mathfrak{grt}$ contains in turn a free Lie algebra on infinite generators (see also \cite[Section 7.2.3]{merkulov2019grothendieck}).
Therefore $\sDer(\overline{\lie_2^\Q})$ and hence $\sDer(\lie_2)$ are non-solvable.
\end{proof}

\section{String links and strongly invertible knots}\label{sec:string}

The aim of this brief section is simply to recall and clarify the definitions concerning strongly invertible knots, theta curves, string links and Milnor invariants.

We start with the definition of strongly invertible knot. Despite not being a fundamental definition in the next sections, we want to stress that all of this construction originates from the goal to investigate the non-abelianity of $\widetilde{\C}$.
\begin{defn}
    A \emph{strongly invertible knot} is a pair $(K,\rho)$ where $K\subseteq S^3$ is an oriented knot and $\rho\in \Diffeo^+(S^3)$ is an  involution such that $\rho(K)=K$ and $\rho$ reverses the orientation on $K$.
\end{defn}
\begin{defn}
    A \emph{direction} on a strongly invertible knot $(K,\rho)$ is the choice of an oriented half-axis $h$, i.e. one of the two connected components of $\Fix(\rho)\setminus K$.
\end{defn}
\begin{defn}
    We say that two DSI knots $(K_i,\rho_i,h_i)$, $i=0,1$ are \emph{equivariantly concordant} if there exists a smooth properly embedded annulus $C\cong S^1\times I\subset S^3\times I$, invariant with respect to some involution $\rho$ of $S^3\times I$ such that:
\begin{itemize}
    \item $\partial (S^3\times I,C)=(S^3,K_0)\sqcup -(S^3,K_1)$,
    \item $\rho$ is in an extension of the strong inversion $\rho_0\sqcup\rho_1$ on $S^3\times 0\sqcup S^3\times 1$,
    \item the orientations of $h_0$ and $-h_1$ induce the same orientation on the annulus $\Fix(\rho)$, and $h_0$ and $h_1$ are contained in the same component of $\Fix(\rho)\setminus C$.
\end{itemize}
\end{defn}
The \emph{equivariant concordance group} is the set $\widetilde{\C}$ of equivalence classes of directed strongly invertible knots up to equivariant concordance, endowed with the operation induced by the \emph{equivariant connected sum}, which we denote by $\widetilde{\#}$ (see \cite{sakuma,boyle2021equivariant} for details).

\begin{defn}
A \emph{labeled theta curve} is a graph $G$ with two vertices $v_1$ and $v_2$ and three edges $e_1, e_2, e_3$ joining $v_1$ and $v_2$, which are considered to be oriented from $v_1$ to $v_2$.
A \emph{(spatial) theta curve} is an embedding $\theta:G\longrightarrow S^3$.
\end{defn}

In \cite{taniyama1993cobordism} Taniyama introduced a notion of cobordism between theta curves and defined the so called \emph{cobordism group of theta curves} $\Theta$, which is a group with the operation of \emph{vertex connected sum} of theta curves.

\begin{remark}\label{remark:SI_to_theta}
Observe that there exists a natural homomorphism
$$
\pi:\widetilde{\C}\longrightarrow \Theta.
$$
In fact, we can associate with a directed strongly invertible knot $(K,\rho,h)$ the theta curve $\theta$ given by the projection in $S^3/\rho\cong S^3$ of $K\cup \Fix(\rho)$. The theta curve is naturally labeled as follows:
\begin{itemize}
    \item the vertex $v_1$ (resp. $v_2$) is the projection of the initial (resp. final) point of $h$,
    \item the edge $e_1$ is the projection of $h$,
    \item the edge $e_2$ is the projection of $K$,
    \item the edge $e_3$ is the projection of $\Fix(\rho)\setminus h$.
\end{itemize}
\end{remark}

We now proceed with the definition of \emph{string link}, introduced originally in \cite{le1988cobordisme}, which generalizes the notion of braid. These are embeddings of some copies of the interval $I=[0,1]$ in the $3$-disk, more precisely:
\begin{defn}
    A $k$-\emph{string link} is (the image of) a proper embedding:
    \[ \sigma: \coprod_{i=1}^k I_i \hookrightarrow D^2\times I \]
    such that, for some fixed $\{p_1,\dots,p_k\} \subseteq D^2 $, $\sigma|_{I_i}(0) = p_i \times 0$ and $\sigma|_{I_i}(1) = p_i \times 1$ for all $i=1,\dots,k$.
    
    The image of $I_i$ is called the \emph{i-th string} of the string link $\sigma$. We will also refer to the $i$-th string of $\sigma$ writing: $I_{\sigma,i}$.
\end{defn}
\begin{figure}
    \centering
    \includegraphics[width=0.3\linewidth]{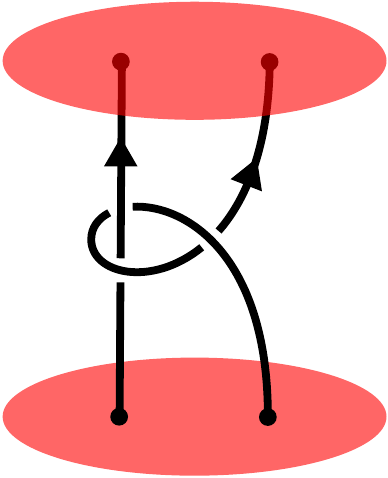}
    \caption{Example of a $2$-string link.}
    \label{fig:stringlink}
\end{figure}
See Figure~\ref{fig:stringlink} for an example. It is clear that each string of a string link inherits an orientation from the standard orientation on the interval. 

Notice that given two $k$-string links $\sigma$ and $\tau$, it is possible to define their \emph{sum} $\sigma \# \tau$. Call $(D\times I)_{\sigma}$ and $(D\times I)_{\tau}$ the two ambient spaces. It is enough to assume that the set of endpoints $\{p_1,\dots,p_k\}\subseteq D$ is shared by the two string links and then identify $(D\times \{1\})_{\sigma}$ with $(D\times \{0\})_{\tau}$.

This sum admits an identity element, given by the trivial string link, that coincides with the trivial $k$-braid. Considering the string links up to isotopies that fix the endpoints, it is easy to notice that the operation of sum is not commutative. 

Now we want to define a suitable notion of \emph{concordance} for string links (see Figure~\ref{fig:concstringlink}). 
\begin{defn}
     We say that two $k$-string links $\sigma$ and $\tau$ are \emph{concordant} if there exist $k$ properly embedded disks $\coprod_{i=1}^k (I\times I)_i \hookrightarrow D\times I\times I$ such that:
     \begin{itemize}
         \item $(I\times 0)_i = I_{\sigma,i}$ for all $i=1,\dots,k$.
         \item $(I\times 1)_i = I_{\tau,i}$ for all $i=1,\dots,k$.
         \item The string link $\coprod_{i=1}^k (0\times I)_i \hookrightarrow D\times 0 \times I$ is the trivial one.
         \item The string link $\coprod_{i=1}^k (1\times I)_i \hookrightarrow D\times 1 \times I$ is the trivial one.
     \end{itemize}
\end{defn}
\begin{comment}
    \begin{figure}
    \centering
     \begin{tikzpicture}
        \draw (0,0) node[above right]{
           \includegraphics[width=0.4\linewidth]{concstringlink.pdf}};
   		\draw (2.25,3.4) node[scale=2,fill=white] (a) {};
        \draw (2.18,3.4) node[scale=2] (a) {$\sigma$};
        \draw (4.15,2.25) node[scale=2] (a) {$\tau$};
         \end{tikzpicture}
    \caption{Concordance of two $2$-strings string links.}
    \label{fig:concstringlink}
\end{figure}
\end{comment}
\begin{figure}
    \centering
     \begin{tikzpicture}
        \draw (0,0) node[above right]{
           \includegraphics[width=0.48\linewidth]{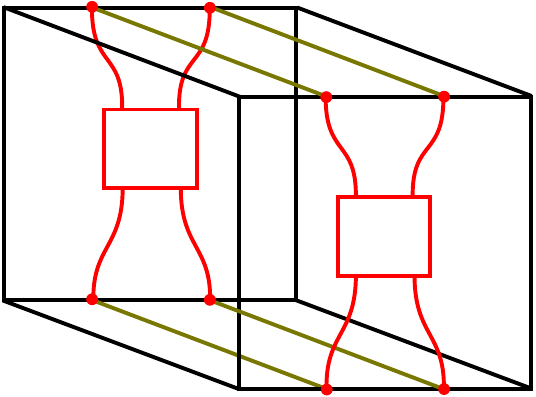}};
        \draw (2.18,3.45) node[scale=2] (a) {$\sigma$};
        \draw (5.33,2.25) node[scale=2] (a) {$\tau$};
         \end{tikzpicture}
    \caption{Concordance of two $2$-string links.}
    \label{fig:concstringlink}
\end{figure}
\iffalse
While on one hand the last two may seem rigid requests, we will see that this definition is exactly what we need. 
\fi
The concordance group $\C(k)$ of $k$-string links is defined as the set of $k$-string links up to concordance, together with the operation of sum $\#$ defined above.

As shown by Le Dimet in \cite{le1988cobordisme}, the group of pure $k$-braids $\mathcal{P}(k)$ is naturally contained in $\C(k)$ as a subgroup.
%l'ultima frase ha forse senso che vada solo nell'introduzione

We now recall briefly the definition of Milnor invariants for string links, following \cite{habegger1998link}. We will focus only on the construction for $2$-string links, since it is the only case we will need in Section \ref{sec:prova}.

Let $\sigma$ be a $2$-string link and let $D_2=D\setminus\{p_1,p_2\}$. 
Denote by $j_0$ and $j_1$ the inclusion of $D_2$ in $D\times I\setminus\sigma$ at time $0$ and $1$ respectively.
Observe that in general $$(j_0)_*, (j_1)_*:\pi_1(D_2)\longrightarrow\pi_1(D\times I\setminus\sigma)$$ are not isomorphisms, but they induce isomorphisms on integral homology. Therefore, by Stalling's Theorem \cite{STALLINGS1965170} we have that $j_0$ and $j_1$ give isomorphisms
$$
(j_0)_*, (j_1)_*:\frac{\pi_1(D_2)}{\pi_1(D_2)_n}\longrightarrow\frac{\pi_1(D\times I\setminus\sigma)}{\pi_1(D\times I\setminus\sigma)_n}
$$
for all $n$, where $\pi_1(X)_n$ is the $n$-th term of the lower central series of $\pi_1(X)$.

Identifying $\pi_1(D_2)$ with the free group $F$ on two generators $x$ and $y$, get an automorphism $A_n(\sigma)=(j_1)_*^{-1}(j_0)_*$ of $F/F_{n+1}$.
As in \cite[Theorem 1.1]{habegger1998link} one can prove that this actually defines a surjective homomorphism
$$
A_n:\C(2)\longrightarrow\Aut_0(F/F_{n+1}),
$$
where $\Aut_0(F/F_{n+1})$ is the subgroup consisting of automorphisms which conjugate $x$ and $y$ and fix the product $xy$.

\begin{remark}\label{remark:total_milnor_invariant}
In the following we will denote by $\overline{F}$ the \emph{algebraic closure} of $F$ in its pro-nilpotent completion $\widehat{F}=\varprojlim F/F_n$ (see \cite{levine1989link} for the definition), and by $\overline{F}_*$ the filtration on $\overline{F}$ given by the lower central series.
As pointed out in \cite[Remark 1.2]{habegger1998link}, Milnor invariants for string links can be gathered to define a \emph{total Milnor invariant}
$$
A:\C(2)\longrightarrow\Aut_0(\overline{F}).
$$
Then, for every $n$ we can retrieve the homomorphism $A_n$ by composition with the projection
$$
\pi_n:\Aut_0(\overline{F})\longrightarrow\Aut_0(\overline{F}/\overline{F}_{n+1}),
$$
since we can identify $F/F_{n+1}\cong\overline{F}/\overline{F}_{n+1}.$
\end{remark}

\section{Solvability of concordance groups}\label{sec:solvab}

The aim of this section is to prove the non-solvability of multiple concordance groups. In particular, we prove that the equivariant concordance group is not solvable and we answer negatively the question regarding the sovability of $\C(n) / \left\langle \left\langle \mathcal{P}(n) \right\rangle\right\rangle$ in \cite[Conjecture 1.3]{kuzbary2023note}.

First of all, consider the following subgroup of $\C(2)$:
\[ \C_u(2) = \left\{ \sigma\in \C(2) \;|\; \text{the first component of } \sigma \text{ is unknotted} \right\}. \]
\begin{thm}\label{thm:Cu_no_sol}
The group $\C_u(2)$ is not solvable.
\end{thm}
In Section \ref{sec:prova} we provide a proof of this result. We now investigate the consequences of Theorem \ref{thm:Cu_no_sol}.
It is straightforward that this implies the following:
\begin{cor}\label{cor:c2_no}
    The group $\C(2)$ of the concordance classes of 2-strands string links is not solvable.
\end{cor}

Consider now $\C_0(2)<\C_u(2)$ the subgroup consisting in string link $\sigma\in\C_u(2)$ with linking number $0$ between components.
\iffalse
 \[\C_0(2) = \left\{ \sigma\in \C_u(2) \;|\; \lk(\sigma)=0 \right\}. \]
Notice that $\C_0(2) < \C_u(2) < \C(2)$.
\fi
Theorem~\ref{thm:Cu_no_sol} implies easily the following:
\begin{cor}\label{cor:C0_no}
    The group $\C_0(2)$ is not solvable.
    \begin{proof}
        Observe that $\C_0(2)$ is a normal subgroup of $\C_u(2)$ and that the quotient $\C_u(2) / \C_0(2)$ is isomorphic to $\mathbb{Z}$. Hence we have the following short exact sequence:
        \[ 1 \rightarrow \C_0(2) \rightarrow \C_u(2) \rightarrow \mathbb{Z} \rightarrow 1. \]
        Since $\mathbb{Z}$ is abelian and $\C_u(2)$ is not solvable, it follows that $\C_0(2)$ cannot be solvable.
    \end{proof}
\end{cor}

We now need a way to deduce the non-solvability of the equivariant concordance group of strongly invertible knots $\widetilde{\C}$ from the results above. The key ingredient is the following result:
\begin{thm}\cite[Proposition 2]{de2005boundary}\label{thm:garoto}
  There is a split extension of groups:
  \[ 1 \rightarrow \mathcal{P}(2) \rightarrow \C(2) \rightarrow \Theta \rightarrow 1. \]
  Here $\mathcal{P}(2)$ is the group of pure braids in $2$ strings, $\C(2)$ are the concordance classes of $2$-strands string links and $\Theta$ is the cobordism group of theta curves.
\end{thm}

\begin{remark}\label{rmk:graziegaroto}
    The existence of the short exact split sequence above allows us to define the following group homomorphism:
    \[ \widetilde{\C} \overset{\pi}{\rightarrow} \Theta \overset{\varphi}{\rightarrow} \C(2).  \]
    The map $\pi$ sends a directed strongly invertible knot $(K,\rho,h)$ to its quotient theta curve as in Remark \ref{remark:SI_to_theta}. The morphism $\phi$ is the section given by Theorem~\ref{thm:garoto}. More precisely, for a theta curve $\theta\in\Theta$ we can always assume that the edge $e_3$ is unknotted and does not undercross nor overcross any other strand. Let $\mathcal{N}(e_3)$ be a tubular neighbourhood small enough not to intersect any crossing on the projection. The homomorphism $\varphi$ sends $\theta\subset S^3$ to $\theta\setminus e_3\subset S^3 \setminus \mathcal{N}(e_3)$, that we can see as a $2$-string link in $D^2\times I$ (see Figure~\ref{fig:varphi}). In order to obtain a well defined map, we must choose appropriately the identification between $S^3 \setminus \mathcal{N}(e_3)$ and $D^2\times I$, which is uniquely determined by requiring the linking number between the two strands of $\theta\setminus e_3$ to be $0$. In Figure \ref{fig:varphi} this corresponds to adding the blue compensatory twists.    

    \begin{figure}[ht]
        \centering
         \begin{tikzpicture}
          \draw (0,0) node[above right]{
        \includegraphics[width=0.7\linewidth]{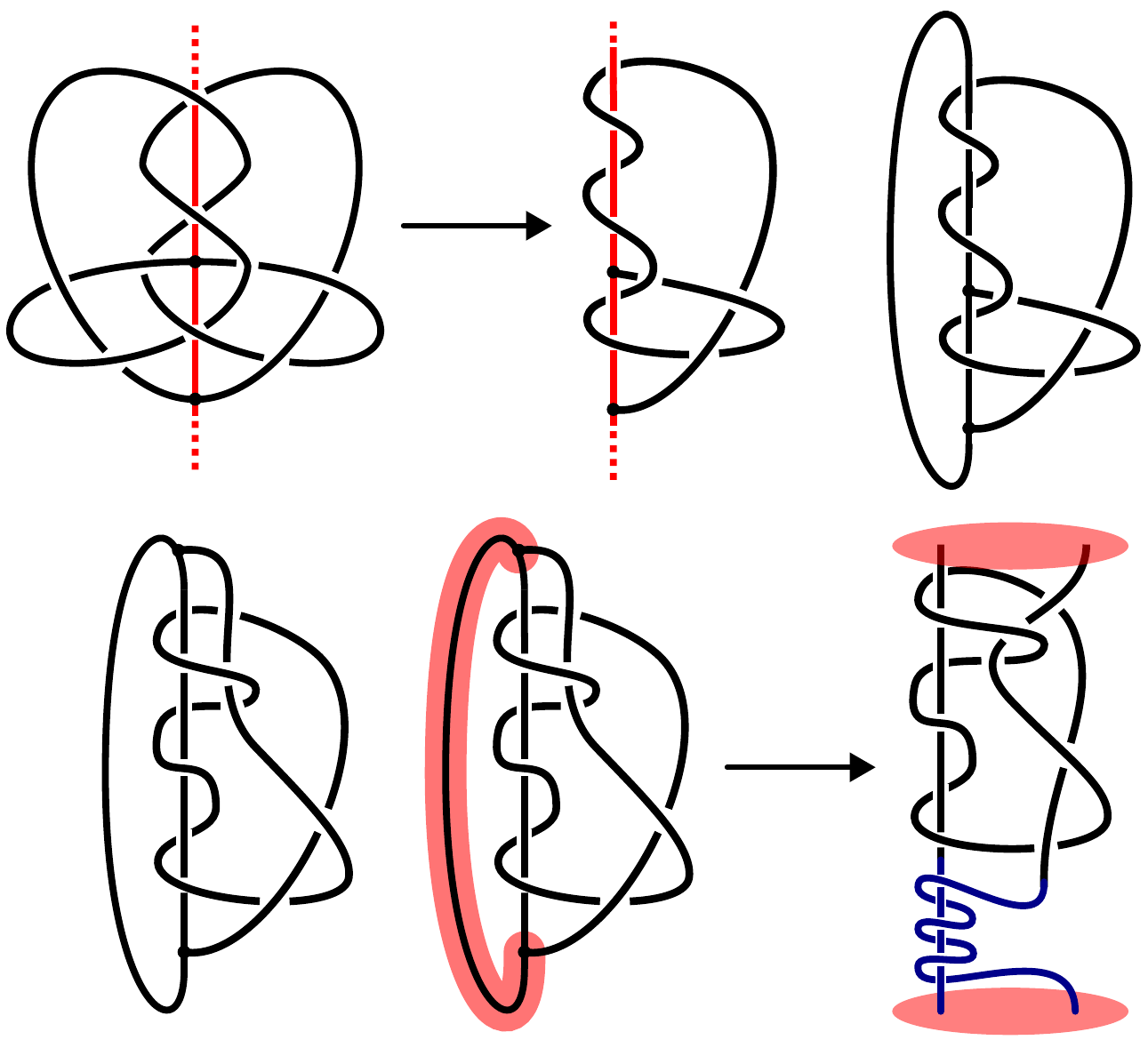}};
        \draw (4.45,8.1) node[scale=2] (a) {$\pi$};
   		\draw (7.49,3.2) node[scale=2] (a) {$\varphi$};
         \end{tikzpicture}
        \caption{An example on the homomorphisms $\pi$ and $\varphi$.}
        \label{fig:varphi}
    \end{figure}

    Observe that the image of the composition $\varphi \circ \pi$ is exactly $\C_0(2)$: 
    \begin{itemize}
        \item[-] $Im(\varphi \circ \pi) \subseteq \C_0(2)$: the image lies obviously in $\C(2)$. Since the first strand is the image of the chosen half axis of the fixed points locus of the strong inversion we also know that the image is contained in $\C_u(2)$. Moreover, by the definition of $\varphi$, the linking number between components is zero, hence $Im(\varphi \circ \pi)$ is contained in $\C_0(2)$.
        \item[-] $\C_0(2) \subseteq Im(\varphi \circ \pi)$: it is enough to glue back $\mathcal{N}(e_3)$ to obtain a theta curve. Observe that for such a theta curve the union of $e_1$ and $e_2$ form an unknot. Hence we get back a directed strongly invertible knot by considering the preimage of $e_2$ in the double cover branched over the unknot $e_1\cup e_3$.
    \end{itemize}    
\end{remark}

This remark allow us to finally show our claims.

\begin{thm}\label{thm:main}
     The equivariant concordance group of strongly invertible knots $\widetilde{\C}$ is not solvable.
\begin{proof}
 In Remark~\ref{rmk:graziegaroto} we defined a  surjective homomorphism $\widetilde{\C} \rightarrow \C_0(2)$. By Corollary~\ref{cor:C0_no} we know that $\C_0(2)$ is not solvable, therefore $\widetilde{\C}$ is not solvable since it has a non-solvable quotient.  
 \end{proof}
\end{thm}

\begin{thm}\label{thm:main.alpha} 
         The quotient group $\C(n) / \langle\langle \mathcal{P}(n) \rangle\rangle$, of the $n$-strands string links over the normal closure of the pure $n$-braids subgroup, is not solvable for any $n\geq 2$.
         In particular, the cobordism group of theta curves $\Theta\cong\C(2) / \mathcal{P}(2)$ is not solvable.
\begin{proof}
 Notice that $\C_0(2) \cap \mathcal{P}(2) = 1$, therefore by Corollary~\ref{cor:C0_no}, we know that $\C(2)/\mathcal{P}(2)$, and hence $\Theta$, are not solvable.
 
 Observe that for $n>2$ we have an injective homomorphism $i_n:\C(2)\longrightarrow\C(n)$, which maps a $2$-string link to an $n$-string link by adding $(n-2)$ trivial strands. The homomorphism $i_n$ has a retraction $s_n:\C(n)\longrightarrow\C(2)$ given by forgetting the last $(n-2)$ strands of the string link. It is not difficult to see that the maps above induce maps on the quotients
 $$
 \C(2)/\mathcal{P}(2) \xrightarrow{\overline{i_n}} \C(n) / \langle\langle \mathcal{P}(n) \rangle\rangle\xrightarrow{\overline{s_n}}\C(2)/\mathcal{P}(2)
 $$
 which show that $\C(2)/\mathcal{P}(2)$ is a subgroup of $\C(n) / \langle\langle \mathcal{P}(n) \rangle\rangle$ for every $n\geq 2$. Therefore $\C(n) / \langle\langle \mathcal{P}(n) \rangle\rangle$ is not solvable.
\end{proof}
\end{thm}

\section{Proof of Theorem~\ref{thm:Cu_no_sol}}\label{sec:prova}
We start with two preliminary lemmas.

\begin{lemma}\label{lemma:surj_milnor}
For every $n$ the restriction of the Artin representation on $\C_u(2)$
$$
A_n:\C_u(2)\longrightarrow\Aut_0(F/F_{n+1})
$$
is surjective.
\end{lemma}
\begin{proof}
The case $n=1$ is trivial, since $\Aut_0(F/F_2)\cong\Z$ and given $\sigma\in\C(2)$ we have that $A_1(\sigma)$ corresponds to the linking number between the components of $\sigma$.
We proceed by induction on $n$. Consider the following commutative diagram (see \cite[Theorem 1.1]{habegger1998link} for the notation and details), where the bottom row is exact:
\begin{center}
    \begin{tikzcd}
        &&\C_u(2)\ar[d,"A_{n+1}",swap]\ar[rd,"A_n"]&&\\
        1\ar[r]&K_n\ar[r]&\Aut_0(F/F_{n+2})\ar[r]&\Aut_0(F/F_{n+1})\ar[r]&1.
    \end{tikzcd}
\end{center}
Suppose by induction that the restriction of $A_n$ on $\C_u(2)$ is surjective. In order to prove the lemma it is sufficient to show that for every $\alpha\in K_n$ there exists $\sigma\in\C_u(2)$ such that $A_{n+1}(\sigma)=\alpha$.
Let $\tau\in\C(2)$ be a string link such that $A_{n+1}(\tau)=\alpha$, and denote by $\widehat{\tau}$ its closure.
Then by \cite[Lemma 3.7]{habegger1998link}, we have that all Milnor invariants of $\widehat{\tau}$ of length $\leq n$ vanish.
By \cite[Theorem 3.3]{cochran1991} we know that there exists a link $L$ with unknotted components with the same Milnor invariants of $\widehat{\tau}$ up to length $n+1$. Let now $\sigma\in\C_u(2)$ be any string link with closure $L$. It follows from \cite[Corollaries 3.6 and 3.8]{habegger1998link} that $A_{n+1}(\sigma)=A_{n+1}(\tau)$. Therefore $A_{n+1}:\C_u(2)\longrightarrow\Aut_0(F/F_{n+2})$ is surjective.
\end{proof}

Let $\{\overline{F}_k\}_{k\geq1}$ be the filtration on $\overline{F}$ (see Remark \ref{remark:total_milnor_invariant}) given by the lower central series, and let $A:\C_u(2)\longrightarrow\Aut_0(\overline{F})$ be total Milnor invariant.

Recall from Lemma \ref{lemma:induced_filtration} that $\{\overline{F}_k\}_{k\geq1}$ induces a filtration on $\C_u(2)$ via $A$, which is given by
$$
\C_u(2)_i=\{\sigma\in\C_u(2)\;|\; \forall j\geq0,\;[\sigma,\overline{F}_j]\subset\overline{F}_{i+j}\}.
$$

\begin{lemma}\label{lemma:filtr_ker}
For all $n\geq 0$, we have that $\C_u(2)_n=\ker(A_n)$.
\end{lemma}
\begin{proof}
Observe that since $A_n=\pi_n\circ A$, we have that $\ker(A_n)=\{\sigma\in\C_u(2)\;|\;[\sigma,\overline{F}_1]\subset\overline{F}_{n+1}\}$, and hence $\C_u(2)_n\subset\ker(A_n)$.
Viceversa, given $\sigma\in\ker(A_n)$ we prove by induction on $j\geq1$ that $[\sigma,\overline{F}_j]\subset\overline{F}_{n+j}$.
First of all, observe that given $a,b\in\overline{F}_{j+1}$ such that $[\sigma,a],[\sigma,b]\in\overline{F}_{n+j+1}$, then $[\sigma,ab]\in\overline{F}_{n+j+1}$. In fact, we can write $[\sigma,ab]=\sigma(a)\sigma(b)b^{-1}a^{-1}=[\sigma,a][\sigma,b][[\sigma,b]^{-1},a]$.
Since $\overline{F}_{j+1}=[\overline{F}_1,\overline{F}_j]$, it is sufficient to prove that given $a\in\overline{F}_1$ and $b\in\overline{F}_j$, we have $[\sigma,[a,b]]\in\overline{F}_{n+j+1}$.
An easy computation shows that $[\sigma,[a,b]]=[[\sigma,a],b][a,[\sigma,b]]$, up to elements in $\overline{F}_{n+j+2}$.
\end{proof}

Let $
\mathcal{L}(\C_u(2)_*)$ be the graded Lie ring given by the filtration on $\C_u(2)$.
Denote by $J_A:\mathcal{L}(\C_u(2)_*)\longrightarrow\Der(\lie_2)$ the Johnson morphism, where we identify $\mathcal{L}(\overline{F}_*)$ with the free Lie ring $\lie_2$ with two generators of degree $1$.

Since for every $\sigma\in\C_u(2)$, the automorphism $A(\sigma)$ acts on $\overline{F}$ by conjugating the generators $x$ and $y$ and preserving the product $xy$, we have that the image of $J_A$ is actually contained in $\sDer(\lie_2)$ (see Definition \ref{def:sder}).

\begin{proof}[Proof of Theorem \ref{thm:Cu_no_sol}]
It follows from Lemma \ref{lemma:group_solv} that it is sufficient to show that $\mathcal{L}(\C_u(2)_*)$ is not solvable.
In order to do so, we prove that the Johnson morphism $$J_A:\mathcal{L}(\C_u(2)_*)\longrightarrow\sDer(\lie_2)$$ is an isomorphism of Lie rings. 
Consider the following commutative diagram, with exact rows.
\begin{center}
    \begin{tikzcd}
        1\ar[r]&\C_u(2)_n/\C_u(2)_{n+1}\ar[r]\ar[d]&\C_u(2)/\C_u(2)_{n+1}\ar[r]\ar[d]&\C_u(2)/\C_u(2)_{n}\ar[r]\ar[d]&1\\
        1\ar[r]&K_n\ar[r]&\Aut_0(F/F_{n+2})\ar[r]&\Aut_0(F/F_{n+1})\ar[r]&1.
    \end{tikzcd}
\end{center}
Observe from the proof of \cite[Theorem 1.1]{habegger1998link} that we can identify $K_n$ with $\sDer(\lie_2)_n$. One can check that with this identification, the map $\C_u(2)_n/\C_u(2)_{n+1}\longrightarrow K_n$ coincides with the restriction of the Johnson homomorphism in degree $n$.
Finally $\C_u(2)_n/\C_u(2)_{n+1}\longrightarrow K_n$ is an isomorphism, since we know from Lemma \ref{lemma:surj_milnor} and Lemma \ref{lemma:filtr_ker} that the central and right map are isomorphism.

Finally, by Proposition \ref{prop:solv_sder} we know that $\sDer(\lie_2)$ is not solvable, therefore $\C_u(2)$ is not a solvable group.
\end{proof}

\section*{Acknowledgements}
We would like to thank Miriam Kuzbary for the helpful conversations about string links and Milnor invariants, and for her encouraging feedback.

\bibliographystyle{alpha} % We choose the "plain" reference style
\bibliography{refs} % Entries are in the refs.bib file
\end{document}